\documentclass[12pt]{amsart}
\usepackage{amscd,amsmath,amsthm,amssymb}
\usepackage{pstcol,pst-plot,pst-3d}%\usepackage[T1]{fontenc}
\usepackage{color}
\usepackage{pstricks}
\usepackage{stmaryrd}
\usepackage{tikz}
\usepackage{lineno}

\newpsstyle{fatline}{linewidth=1.5pt}
\newpsstyle{fyp}{fillstyle=solid,fillcolor=verylight}
\definecolor{verylight}{gray}{0.97}
\definecolor{light}{gray}{0.9}
\definecolor{medium}{gray}{0.85}
\definecolor{dark}{gray}{0.6}

 \def\MG{{\mathcal G}}

 \def\G{{\mathcal G}}

 \def\ab{{\mathbf a}}

 \def\opn#1#2{\def#1{\operatorname{#2}}} % to make operators
 \opn\chara{char} \opn\length{\ell} \opn\pd{pd} \opn\rk{rk}
 \opn\projdim{proj\,dim} \opn\injdim{inj\,dim} \opn\rank{rank}
 \opn\depth{depth} \opn\grade{grade} \opn\height{height}
 \opn\embdim{emb\,dim} \opn\codim{codim}
 
 \opn\Tr{Tr} \opn\bigrank{big\,rank}
 \opn\superheight{superheight}\opn\lcm{lcm}
 \opn\trdeg{tr\,deg}%\emph{
 \opn\reg{reg} \opn\lreg{lreg} \opn\ini{in} \opn\lpd{lpd}
 \opn\size{size} \opn\sdepth{sdepth}
 \opn\link{link}\opn\fdepth{fdepth}\opn\lex{lex}
 \opn\tr{tr}
 \opn\type{type}
 \opn\gap{gap}
 \opn\div{div} \opn\Div{Div} \opn\cl{cl} \opn\Cl{Cl}
 \opn\Spec{Spec} \opn\Supp{Supp} \opn\supp{supp} \opn\Sing{Sing}
 \opn\Ass{Ass} \opn\Min{Min}\opn\Mon{Mon}
 \opn\Ann{Ann} \opn\Rad{Rad} \opn\Soc{Soc}
 \opn\Im{Im} \opn\Ker{Ker} \opn\Coker{Coker} \opn\Am{Am}
 \opn\Hom{Hom} \opn\Tor{Tor} \opn\Ext{Ext} \opn\End{End}
 \opn\Aut{Aut} \opn\id{id}
 
 \opn\nat{nat}
 \opn\pff{pf}%   \pf exists already
 \opn\Pf{Pf} \opn\GL{GL} \opn\SL{SL} \opn\mod{mod} \opn\ord{ord}
 \opn\Gin{Gin} \opn\Hilb{Hilb}\opn\sort{sort}
 \opn\PF{PF}\opn\Ap{Ap}
 \opn\aff{aff}
 \opn\relint{relint} \opn\st{st}
 \opn\lk{lk} \opn\cn{cn} \opn\core{core} \opn\vol{vol}  \opn\inp{inp} \opn\nilpot{nilpot}
 \opn\link{link} \opn\star{star}\opn\lex{lex}\opn\set{set}
 \opn\width{wd}
 \opn\Fr{F}
 \opn\QF{QF}
 \opn\G{G}
 \opn\type{type}\opn\res{res}
 \opn\conv{conv}
 \opn\gr{gr}
 
 \def\pot#1#2{#1[\kern-0.28ex[#2]\kern-0.28ex]}

 \opn\dirlim{\underrightarrow{\lim}}
 \opn\inivlim{\underleftarrow{\lim}}
 \let\to=\rightarrow
 
 \def\Implies{\ifmmode\Longrightarrow \else
         \unskip${}\Longrightarrow{}$\ignorespaces\fi}
 \def\implies{\ifmmode\Rightarrow \else
         \unskip${}\Rightarrow{}$\ignorespaces\fi}
 \def\iff{\ifmmode\Longleftrightarrow \else
         \unskip${}\Longleftrightarrow{}$\ignorespaces\fi}

 \let\:=\colon
 \newtheorem{Theorem}{Theorem}[section]
 \newtheorem{Lemma}[Theorem]{Lemma}

 \newtheorem{Example}[Theorem]{Example}
 
 \newtheorem{Definition}[Theorem]{Definition}

 \let\epsilon\varepsilon
 \let\kappa=\varkappa
 \def\qed{\ifhmode\textqed\fi
       \ifmmode\ifinner\quad\qedsymbol\else\dispqed\fi\fi}
 \def\textqed{\unskip\nobreak\penalty50
        \hskip2em\hbox{}\nobreak\hfil\qedsymbol
        \parfillskip=0pt \finalhyphendemerits=0}
 \def\dispqed{\rlap{\qquad\qedsymbol}}
 
 \opn\dis{dis}
 \def\pnt{{\raise0.5mm\hbox{\large\bf.}}}
 
 \opn\Lex{Lex}

\begin{document}
%\linenumbers
\title {Sortable Freiman ideals }

\author {J\"urgen Herzog and  Guangjun Zhu$^{^*}$}

\address{J\"urgen Herzog, Fachbereich Mathematik, Universit\"at Duisburg-Essen, Campus Essen, 45117
Essen, Germany} \email{juergen.herzog@uni-essen.de}

\address{Guangjun Zhu, School of Mathematical Sciences, Soochow University,
 Suzhou 215006, P. R. China. }
\email{zhuguangjun@suda.edu.cn(Corresponding author:Guangjun Zhu)}

\dedicatory{ }

\begin{abstract}
In this paper it is shown that a sortable ideal $I$ is Freiman if and only if its sorted graph is chordal. This characterization is used to give
a complete classification of  Freiman principal Borel  ideals and of Freiman Veronese type ideals with constant bound.
\end{abstract}

\thanks{* Corresponding author. }

\subjclass[2010]{Primary 13C99; Secondary 13H05, 13H10.}
%		13H10   	Special types (Cohen-Macaulay, Gorenstein, Buchsbaum, etc.)
%		13D02   	Syzygies, resolutions, complexes
%		05E40   	Combinatorial aspects of commutative algebra
%		16S36   	Ordinary and skew polynomial rings and semigroup rings

%		14M25   	Toric varieties, Newton polyhedra [See also 52B20]
%		13A02   	Graded rings
%		13F20   	Polynomial rings and ideals; rings of integer-valued polynomials
%		13A18   	Valuations and their generalizations
%		06A11   	Algebraic aspects of posets

\keywords{Freiman ideal, sorted ideal,  principal Borel  ideal, Veronese type ideals with constant bound}

\maketitle

\setcounter{tocdepth}{1}
%\tableofcontents

\section*{Introduction}
The concept  of Freiman ideals appeared the first time in \cite{HZ}. Based on a famous theorem from additive number theory, due to  Freiman
\cite{F1},  it was proved  in  \cite{HMZ} that if $I\subset S$ is an equigenerated  monomial ideal,  then  $\mu(I^2)\geq
\ell(I)\mu(I)-{\ell(I)\choose 2}$. Here $\mu(I)$ denotes the least number of generators of
the ideal $I$  and $\ell(I)$ the analytic spread of $I$.  The assumption that  $I$ be equigenerated  is  essential. Indeed, in \cite{EHM} it has
been shown that for every integer $m\geq 6$, there exists a monomial ideal $I\subset  K[x,  y]$  such that $\mu(I) = m$  and $\mu(I^2) = 9$.

In \cite{HZ}  an equigenerated monomial $I$ is called {\em Freiman} if $\mu(I^2)= \ell(I)\mu(I)-{\ell(I)\choose 2}$. In the same paper we
characterized  Freiman ideals for  special classes of  principal Borel ideals and Veronese type ideals. One purpose of this paper is to give a
full classification of Freiman ideals for all   principal Borel  ideals and all Veronese type ideals with constant bound, thereby generalizing the results of
\cite{HZ}.

In \cite{HHZ}, subsequent to paper \cite{HZ}, a more systematic treatment of Freiman ideals was presented. Among several other characterizations
of the Freiman property, it was shown in \cite{HHZ} that $I$ is a Freiman ideal if and only if the fiber cone of $I$ has a $2$-linear resolution.
This is a particularly nice property which  implies that the defining ideal of the  fiber cone of $I$ is generated by quadrics. In the present
paper we use this characterization of Freiman ideals.

There is a remarkable property of equigenerated monomial ideals, called sortability, which guarantees that the defining relations of their fiber
cone is generated by the so-called sorting relations. The sorting relations are quadratic binomials and form a Gr\"obner basis with respect to a
suitable monomial order. Sortable sets of monomials  have been introduced by Sturmfels \cite{St}.  Some basic properties  of sortable ideals are
discussed  in  \cite{EH}. In Section $1$ we  associated with a sortable ideal $I$  a finite simple graph $G_{I,s}$, whose edges correspond to
the sortable pairs of the monomial generators of $I$, and prove in  Theorem~\ref{main}  that a sortable ideal $I$ is Freiman, if and only if
$G_{I,s}$ is chordal.  This will be our main tool to give in the following sections  a complete characterization of Freiman  principal Borel
ideals and of Freiman ideals of Veronese type with constant bound.

\section{Sortability and the Freiman property}
%\label{1}

The aim of this section is to analyze what it means that a sortable  monomial ideal is  a Freiman ideal.

Let $K$ be a field and $S=K[x_1,\ldots,x_n]$ be the polynomial ring over $K$ in $n$ indeterminates.
Let $d$ be a positive integer, $S_d$ the $K$-vector space generated by the monomials of degree $d$ in $S$, and take two monomials $u,v\in S_d$. We
write $uv=x_{i_1}x_{i_2}\cdots x_{i_{2d}}$ with $1\leq i_1\leq i_2\leq \cdots \leq i_{2d}\leq n$,  and define
$$u'=x_{i_1}x_{i_3}\cdots x_{i_{2d-1}},\quad \text{and} \quad v'=x_{i_2}x_{i_4}\cdots x_{i_{2d}}.$$
The pair $(u',v')$ is called the {\it sorting} of $(u,v)$.

In this way we obtain a map
\[
\sort: S_d\times S_d \to S_d\times S_d,\ (u, v)\mapsto (u', v').
\]

This map is called the sorting operator.
For example, if $u=x_1x_3^{2}$
and $v=x_2^2x_3$, then $\sort(u,v)= (x_1x_2x_3, x_2x_3^2)$. A pair $(u,v)$ is called to be  {\em sorted} if $\sort(u,v)=(u,v)$ or
$\sort(u,v)=(v,u)$, otherwise it is called to {\em unsorted}. Notice that $\sort(u,v)=\sort(v,u)$.

\begin{Definition} \label{HHF}{\em A subset $B\subset S_d$ of monomials is called {\em sortable} if
$\sort(B\times B)\subset B\times B$.}
\end{Definition}

Let $I$ be an equigenerated monomial ideal. The unique minimal set of monomial generators of $I$ is denoted by $G(I)$. We call $I$ a {\em sortable
ideal}, if $G(I)$ is a sortable set. The defining ideal $J$ of the fiber cone $F(I)$ of a sortable ideal $I$ has very nice Gr\"obner basis.
Indeed, let $R$ be the polynomial ring $K[t_v\; | \; v \in G(I)]$, and $\phi: R \to F(I)$ be the $K$-algebra homomorphism which maps $t_v$ to $v$
for all $v \in G(I)$. We denote by $J$ the kernel of $\phi$. Then $F(I)=R/J$. Obviously, the quadratic binomials
$f_{(u,v)}=t_ut_v-t_{u^\prime}t_{v^\prime}$ with $(u,v)$ unsorted belong to $J$.

In \cite{DN}, De Negri noticed the following fact which easily follows from a theorem of
Sturmfels \cite{St} (see also  \cite[Theorem 6.16]{EH}), and  which is crucial for this paper.

\begin{Theorem}\label{algebra}
Let $I$ be a sortable ideal. Then there exists a monomial order $<$ on $R$ such that $\ini_<(f_{(u,v)})=t_ut_v$ for all unsorted pairs $(u,v)$
with $u,v\in G(I)$. Moreover, the set of binomials $\MG=\{f_{(u,v)}\:\; (u,v)\text{ is unsorted }\}$ is a Gr\"obner basis of the toric ideal $J$.
\end{Theorem}

With a sortable ideal $I$ we attach two graphs, namely the {\em sorted graph} $G_{I,s}$ of $I$,  and the {\em  unsorted graph} $G_{I,u}$ of $I$.
The vertex sets of  both graphs is  $G(I)$. Its edge sets are
\[
E(G_{I,s})=\{\{u,v\}\: \text{$u\neq v$, $(u,v)$  is sorted}\},
 \]
and
\[
E(G_{I,u})=\{\{u,v\}\: \text{$u\neq v$, $(u,v)$  is unsorted}\}.
\]
Theorem~\ref{algebra} implies that $\ini_<(J)$ is the edge ideal of $G_{I,u}$, where $J$ is the defining toric ideal of the fiber cone $F(I)$ of
$I$.

%Let $W\subset G(I)$ be a subset of the vertex set of $G_{I,s}$, and let $\Gamma=(G_{I,s})_W$ be the corresponding induced subgraph of $G_{I,s}$.
Then %$\Gamma$ is the sorted graph of the ideal $(u\:\; u\in W)$.

For a graph $G$ we denote by $G^c$ the {\em complementary graph} of $G$,  which is the graph with the same vertex set as that of $G$ and with
\[
E(G^{c})=\{\{u,v\}\\:\; \{u,v\}\not\in E(G)\}.
\]
Obviously one has $(G_{I,u})^c=G_{I,s}$.

\medskip
%Let $I$ be any  equigenerated monomial ideal, and  $h=(1,h_1,h_2,\ldots)$  the $h$-vector of $F(I)$. As  a  consequence of a theorem of Freiman
\cite{F1} one has that $h_2\geq 0$, see \cite[Corollary 2.6]{HMZ}. The ideal $I$ is called a {\em Freiman ideal}  if $h_2=0$.

Freiman ideals are  characterized by the following result shown in  \cite[Theorem 1.3]{HHZ}.

\begin{Theorem}
\label{characterize}
Let $I$ be any  equigenerated monomial ideal. Then $I$ is a Freiman ideal if and only if $F(I)$ is Cohen-Macaulay and the defining ideal of $F(I)$
has a $2$-linear resolution.
\end{Theorem}

Now we are ready to prove the main result of this section.

\begin{Theorem}
\label{main}
Let $I$ be a sortable ideal. Then $I$ is  Freiman  if and only if the sorted graph $G_{I,s}$ of $I$ is chordal. In particular, if $G_{I,s}$
contains an induced  $t$-cycle with $t\geq 4$, then $I$ is not Freiman.
\end{Theorem}

\begin{proof}
Let $I$ be sortable,  and let $J$ be the defining toric ideal of $F(I)$.  By Theorem~\ref{algebra}, $\ini_<(J)$ is squarefree. By \cite{St} (see
also
\cite[Corollary 4.26]{HHO}) this implies that $F(I)$ is normal. Then by a theorem of Hochster \cite{Ho} it follows that $F(I)$ is Cohen-Macaulay.
As observed before, $\ini_<(J)$ is the edge ideal of $G_{I,u}$ and that $G_{I,u}^c=G_{I,s}$. Now we apply a famous and well-known theorem of
Fr\"oberg \cite{F}   and deduce that $\ini_<(J)$ has a $2$-linear resolution, due to the fact that  $G_{I,s}$ is chordal by assumption. This
implies that $J$ has $2$-linear resolution as well, see for example \cite[Theorem 3.3.4]{HH}. Thus $I$ is Freiman by Theorem~\ref{characterize}.

Conversely, assume that $I$ is Freiman. Then $J$ has a $2$-linear resolution. Since $\ini_<(J)$ is squarefree, it follows from \cite[Corollary
2.7]{CV} that $\ini_<(J)$ has a $2$-linear resolution as well. Again applying the theorem of Fr\"oberg it follows that $G_{I,s}$  is chordal.
\end{proof}

In the following sections we apply this theorem to special classes of monomial ideals.

\section{Freiman principal Borel ideals}

A monomial ideal $I\subset S$ is called {\em strongly stable}, if for all $v\in G(I)$ and all $j\in \supp(v)$ it follows that $x_i(v/x_j)\in  I$
for all $i<j$.  Given monomials $v_1,\ldots,v_m\in S$,  there exists a unique smallest strongly stable ideal, denoted $B(v_1,\ldots,v_m)$,  which
contains  monomials $v_1,\ldots,v_m$. The monomial ideal $I$ is called a {\em principal Borel ideal}  if
$I=B(v)$ for some monomial $v\in S$.

\begin{Example}
\label{National Day}{\em
Let   $I=B(v)$, where  $v=x_3^2$, then  $G(I)=\{x_1^2,x_1x_2,x_1x_3,x_2^2,x_2x_3,x_3^2\}$.  The edge set of the  sorted graph  of $I$ is
\begin{eqnarray*}
&&\{\{x_1^2,x_1x_2\},\{x_1^2,x_1x_3\},\{x_1x_2,x_1x_3\}, \{x_1x_2,x_2^2\},\{x_1x_2,x_2x_3\},\\
&&\{x_1x_3,x_2x_3\},\{x_1x_3,x_3^2\}, \{x_2^2,x_2x_3\},\{x_2x_3,x_3^2\}\}.
\end{eqnarray*}
}
\end{Example}

The following theorem characterizes all Freiman ideals among the principal Borel ideals.

\begin{Theorem}\label{ $0$-spreadprincipalBore}
 Let $d\geq 2$ and $n\geq 3$ be two integers, and  $S=K[x_1,\ldots,x_n]$ the polynomial ring over $K$ in $n$ indeterminates. Let
 $u$ be a monomial of degree $d$ in $S$. Then

 \begin{enumerate}
\item[(a)]   If $d=2$, then $B(u)$ is  Freiman if and only if $u$ is a minimal monomial generator of $(x_1,x_2,x_3)^2$, or
    $x_1(x_4,\ldots,x_n)$, or
 $x_2(x_4,\ldots,x_n)$.
\item[(b)] If $d=3$, then $B(u)$ is  Freiman  if and only if   $u$  is a minimal monomial generator of $x_1(x_1,x_2,x_3)^2$, or
    $x_1(x_1,x_2)x_i$ with $i>3$, or $x_2^2(x_2,\ldots,x_n)$.
 \item[(c)] If $d\geq 4$, then  $B(u)$ is  Freiman  if and only if  $u=x_{1}^{d-2}x_{3}^2$,  or $u$  is a minimal  monomial  generators of
     $x_1^{d-1}(x_1,\ldots,x_n)$, or  $x_{1}^{d-r-1}x_{2}^{r}(x_i,\ldots,x_n)$ where $1\leq r\leq d-1$ and $i\geq 2$.
  \end{enumerate}
  \end{Theorem}
  \begin{proof}
(a) is shown in  \cite[Theorem 4.2 (a)]{HZ}.

(b) Let   $v\in S$ be a monomial. Then
\begin{eqnarray}
\label{fact}
\text{$B(v)$ is Freiman if and only if $B(x_1^kv)$ is Freiman.}
\end{eqnarray}
 We will use this fact several times in the proof. By applying (\ref{fact})  and the result in (a), we obtain that
$B(u)$ is  Freiman  if  $u$ is a minimal monomial generator of $x_1(x_1,x_2,x_3)^2$, or $x_1(x_1,x_2)x_i$ with $i>3$, and
by  \cite[Theorem 4.2 (b)]{HZ}, we see that $B(u)$ is  Freiman  if  $u$  is a minimal monomial generator of $x_2^2(x_2,\ldots,x_n)$.

It remains to be shown that $B(u)$ is not Freiman if $u$ is different from the monomials listed (b). Then  $u$ is a minimal monomial generator of
the following ideals:
(i) $x_1(x_4,\ldots,x_n)^2$, or (ii) $x_1x_3(x_4,\ldots,x_n)$, or (iii) $x_2(x_3,\ldots,x_n)^2$, or (iv) $(x_3,\ldots,x_n)^3$.
By (\ref{fact})  we have  to show that $B(u)$ is not Freiman if  $u$ is a minimal monomial generator of
the following ideals:

(i) $(x_4,\ldots,x_n)^2$, or (ii) $x_3(x_4,\ldots,x_n)$, or (iii) $x_2(x_3,\ldots,x_n)^2$, or (iv) $(x_3,\ldots,x_n)^3$.

For case (i) or (ii),  $B(u)$ is not Freiman by (a).
For case (iii) or (iv),  the elements $u_1=x_1x_2^2$, $u_2=x_1^2x_2$, $u_3=x_1^2x_3$, $u_4=x_1x_3^2$, $u_5=x_2x_3^2$ and  $u_6=x_2^2x_3$ are the
vertices of an induced $6$-cycle of  $G_{B(u),s}$. Indeed, $(u_i,u_{i+1})$ for $i=1,\ldots, 5$  and $(u_1,u_6)$ are sorted pairs, while the pairs
corresponding to the chords of the $6$-cycle are not sorted. This show that it is  is an induced $6$-cycle.
Hence $B(u)$ is not Freiman by Theorem~\ref{main}.

(c) Suppose first that   $u=x_{1}^{d-2}x_{3}^2$,  or  $u$ is a minimal monomial generator of $x_1^{d-1}(x_1,\ldots,x_n)$, or
$x_{1}^{d-r-1}x_{2}^{r}(x_i,\ldots,x_n)$ where $1\leq r\leq d-1$ and $i\geq 2$. By (\ref{fact}) it is enough to consider the  case $u=x_{3}^2$,
or $u$ is a minimal monomial generator of $(x_2,\ldots,x_n)$, or  $x_{2}^{r}(x_i,\ldots,x_n)$ where $1\leq r\leq d-1$ and $i\geq 2$. It follows
from  \cite[Theorem 3.3 and Theorem 4.2]{HZ} that   $B(u)$ is  Freiman.

Now we will  prove that $B(u)$ is not Freiman if $u$ is different from the monomials listed (c). Then $u$ is a minimal monomial generator of
$x_1^{a}x_2^{b}(x_3,\ldots,x_n)^{d-(a+b)}$ where $0\leq a+b\leq d-2$. By applying (\ref{fact}) we have to show  that if  $u$ is a minimal monomial
generator of the following ideals:
 (i) $x_2^{d-\ell}(x_3,\ldots,x_n)^{\ell}$ with $\ell\geq 2$, or  (ii) $x_3(x_3,\ldots,x_n)^{\ell}$ with $\ell=2,\ldots,d-1$, or  (iii)
 $(x_3,\ldots,x_n)^{\ell}$ with $\ell=3,\ldots,d$, then $B(u)$ is not  Freiman.

In all these cases, the elements $x_1^{d-2}x_2^2$, $x_1^{d-1}x_2$, $x_1^{d-1}x_3$, $x_1^{d-2}x_3^2$, $x_1^{d-3}x_2x_3^2$ and  $x_1^{d-3}x_2^2x_3$
are the vertices of an induced $6$-cycle of  $G_{B(u),s}$. Hence $B(u)$ is not Freiman, as desired.

\end{proof}

\section{Freiman ideals of Veronese type with constant bound}

Given positive integers  $n$, $d$ and a sequence $\ab$ of integers $1\leq a_1\leq \cdots \leq a_n\leq d$ with $\sum_{i=1}^na_i>d$, one defines the
monomial ideal  $I_{\ab,n,d}\subset S=K[x_1,\ldots,x_n]$ with
\[
G(I_{\ab,n,d})=\{x_1^{b_1}x_2^{b_2}\cdots x_n^{b_n}\; \mid \; \sum_{i=1}^nb_i=d \text{ and  $b_i\leq a_i$ for $i=1,\ldots,n$}\}.
\]
In this section we only consider the case that  $a_i=k$ for $i=1,\ldots,n$, and denote the corresponding Veronese type ideal with constant bound $k$  by $I_{k,n,d}$.  Note that $I_{k,n,d}=I_{d,n,d}$ if $k>d$. Therefore, we may  always assume that $k\leq d$.

\medskip
\begin{Lemma}
\label{specialk}
The ideal $I_{k,n,d}$ is sortable.
\end{Lemma}
 \begin{proof} Let $u,v\in G(I_{k,n,d})$, and let $\sort(u,v)= (u',v')$. Let $uv=x_{i_1}x_{i_2}\ldots x_{i_{2d}}$. Let $j$ be an integer with $1\leq j\leq n$,  and let $\ell_1\leq \ell_2$ be the integer with the property that $i_\ell=j$ if and only if $\ell_1\leq \ell \leq \ell_2$. Then $\ell_2-\ell_1+1\leq 2k$. Let $u'=x_1^{c_1}\cdots x_n^{c_n}$.  Then $c_j=|\{\ell\:\; \ell_1\leq \ell \leq \ell_2 \text{ and } \ell \text{ is odd}\}|$. It follows that $c_j\leq k$ for all $j$, and this implies that $u'\in G(I_{k,n,d})$, Similarly, $v'\in G(I_{k,n,d})$.
\end{proof}

Since $I_{k,n,d}$ is sortable, we may apply  Theorem~\ref{main} to check which of the ideals $I_{k,n,d}$  are Freiman.

 \medskip
In the following proofs  we will use

\begin{Lemma}
  \label{variable} Let $A=K[x_1,\ldots,x_n]$ and $B=A[y]$ be polynomial rings over a field
$K$. Let $I=(u_1,\ldots, u_m)\subset A$ be a  monomial ideal generated in degree $d$,  and $J\subset B $ a monomial ideal generated in degree $d+1$ with  $(yu_1,\ldots, yu_m)\subset J$.  If $I$ is not Freiman, then $J$ is not Freiman.

 \end{Lemma}
  \begin{proof} Since $I$ is not Freiman,  there exists  an induced   $t$-cycle $C_t$ with $t\geq 4$ of  $G_{I,s}$. Let  $\{u_{i_1},\ldots, u_{i_t}\}$ be the vertex set of $C_{t}$, then  $yu_{i_1},\ldots, yu_{i_t}\in J$. It follows that the
  elements   $yu_{i_1},\ldots, yu_{i_t}$  form the vertices of an induced  $t$-cycle of $G_{J,s}$.
  Therefore, $J$ is not Freiman.
 \end{proof}

The following theorem classifies the Freiman ideals of the form $I_{1,n,d}$.

\begin{Theorem}
\label{specialk}
The ideal $I_{1,n,d}$ is Freiman, if and only if one of the following conditions holds:
 \begin{enumerate}
\item[(a)] $n=2$ and  $d=1$.
\item[(b)] $n\geq 3$ and  $d=1$, or $d=n-1$.
\end{enumerate}
\end{Theorem}

 \begin{proof} If $d=1$ or $d=n-1$, then the sorted graph $G_{I_{1,n,d},s}$  is a complete graph. Thus   Theorem  \ref{main} implies that $I_{1,n,d}$  is  Freiman.

It remains to be shown that $I_{1,n,d}$  is not Freiman if $n\geq 4$ and $2\leq d\leq n-2$. In this case, the elements  $(\prod\limits_{j=4}^{d+1}x_{n-j})x_{n-3}x_{n-2}$, $(\prod\limits_{j=4}^{d+1}x_{n-j})x_{n-3}x_n$,
$(\prod\limits_{j=4}^{d+1}x_{n-j})x_{n-1}x_n$, and $(\prod\limits_{j=4}^{d+1}x_{n-j})x_{n-2}x_{n-1}$,  are the vertices of an induced  $4$-cycle of  $G_{I_{1,n,d},s}$,  where  $\prod\limits_{j=4}^{d+1}x_{n-j}=1$ if $d=2$.
It follows that $I_{1,n,d}$ is not Freiman.
\end{proof}

Next we consider the case $k=2$.

\begin{Theorem}
\label{k=2}
The ideal $I_{2,n,d}$ is Freiman if and only if one of the following conditions holds:
 \begin{enumerate}
\item[(a)] $n=2$ and  $d=2$, or $d=3$.
\item[(b)] $n=3$ and  $d=2$, or $d=4$, or $d=5$.
\item[(c)] $n\geq 4$ and  $d=2n-1$.
\end{enumerate}
\end{Theorem}

 \begin{proof} (a ) By a simple computation shows   that  $I_{2,n,d}$  is  Freiman if $d=2$ or $d=3$.

 (b)  Since the graphs $G_{I_{2,3,2},s}$, $G_{I_{2,3,4},s}$ and
 $G_{I_{2,3,5},s}$ are chordal, as can be seen easily,  the ideals  $I_{2,3,2}$, $I_{2,3,4}$ and $I_{2,3,5}$  are  Freiman.

 If $d=3$, then the elements  $x_1x_2^2,x_1^2x_2,x_1^2x_3,x_1x_3^2,x_2x_3^2$ and $x_2^2x_3$  are the vertices of an induced  $6$-cycle of  $G_{I_{2,3,3},s}$. It follows that $I_{2,3,3}$ is not Freiman.

(c) If $d=2n-1$, then
$I_{2,n,d}=(x_1\prod\limits_{j=2}^{n}x_j^2,\ldots,x_i\prod\limits_{\begin{subarray}{l}
j=1\\
j\neq i
\end{subarray}}^{n}x_j^2,\ldots,x_n\prod\limits_{j=1}^{n-1}x_j^2)$. Thus  $G_{I_{2,n,d},s}$
is a complete graph. Hence $I_{2,n,d}$  is  Freiman.

Next we prove that $I_{2,n,d}$ is not  Freiman if $d\neq 2n-1$.
We apply  induction on $n$.
First, we consider the case $n=4$. Then $d=2,3,4,5$ or $6$, If  $d=2$, then the elements $x_1x_2$, $x_1x_4$, $x_3x_4$ and  $x_2x_3$ are the vertices of an induced $4$-cycle of  $I_{2,4,2}$. If $d=3$, then the elements $x_1^2x_2$, $x_1^2x_4$, $x_1x_3x_4$ and  $x_1x_2x_3$ are the vertices of an induced $4$-cycle of  $I_{2,4,3}$. If $d=4$, then  the elements $x_1^2x_2^2$, $x_1^2x_2x_4$, $x_1x_2x_3x_4$ and  $x_1x_2^2x_3$ are the vertices of an induced $4$-cycle of  $I_{2,4,4}$. If $d=5$, then the elements $x_1^2x_2^2x_3$, $x_1^2x_2x_3x_4$, $x_1x_2x_3^2x_4$ and  $x_1x_2^2x_3^3$ are the vertices of an induced $4$-cycle of  $I_{2,4,5}$. Finally, if $d=6$, then the elements $x_1^2x_2^2x_3x_4$, $x_1^2x_2x_3x_4^2$, $x_1x_2x_3^2x_4^2$ and  $x_1x_2^2x_3^3x_4$ are the vertices of an induced $4$-cycle of  $I_{2,4,6}$.
This shows that  $I_{2,n,d}$ for $d=2,\ldots, 6$ is not Freiman.

Let $n\geq 5$. By the induction hypothesis we know that $I_{2, n-1,d}$ is not Freiman for $d=2,\ldots,2n-4$.  Therefore, by  Lemma~\ref{variable}, it follows that $I_{2, n,d}$ is not Freiman for $d=2,\ldots,2n-4$. Hence remains to be  shown  that $I_{2,n,d}$ is not Freiman for  $d=2n-3$ and $d=2n-2$. The ideal $I_{2,n-1,2n-4}$ is not  Freiman for $n=5$ as shown before, and if $n>5$,  $I_{2,n-1,2n-4}$ is not  Freiman by the induction hypothesis. Hence there exists  an induced   $t$-cycle $C_t$  of  length $t\geq 4$
in  $G_{I_{2,n-1,2n-4},s}$. Let  $\{u_{i_1},\ldots, u_{i_t}\}\subset K[x_1,\ldots,x_{n-1}]$ be the vertex set of $C_{t}$.
If $d=2n-3$, then the elements  $u_{i_1}x_n,\ldots, u_{i_t}x_n$  are the vertices of an  induced $t$-cycle of $G_{I_{2,n,2n-3},s}$.
Hence $I_{2,n,2n-3}$ is not Freiman. Finally,
if $d=2n-2$, then the elements  $u_{i_1}x_n^2,\ldots, u_{i_t}x_n^2$  are the vertices of an  induced $t$-cycle of $G_{I_{2,n,2n-2},s}$.
Hence $I_{2,n,2n-2}$ is not Freiman.
This completes the proof.
\end{proof}

The next theorem treats the case $k\geq 3$.

\begin{Theorem}
\label{k=3}
Let $k\geq 3$. The ideal $I_{k,n,d}$ is Freiman if and only if one of the following conditions holds:
 \begin{enumerate}
\item[(a)] $n=2$ and $d=k,\ldots,2k-1$.
\item[(b)] $n=3$ and $d=3k-2$  or $d=3k-1$.
\item[(c)] $n\geq 4$ and  $d=kn-1$.
\end{enumerate}
\end{Theorem}
\begin{proof} (a) The  graph $G_{I_{2,n,d},s}$
is path graph for $d=k,\ldots,2k-1$. Hence $I_{2,n,d}$ is Freiman.

(b) For $d=3k-2$ the graph  $G_{I_{k,3,d},s}$ has 6 vertices, and for  and $d=3k-1$ it has 4 vertices.  It can be checked that both graphs are chordal.  Hence $I_{k,3,d}$ is Freiman for  $d=3k-2$  or $d=3k-1$.

It remains to be  shown that $I_{k,3,d}$ is not Freiman if  $d=k,\ldots, 3k-3$. If $d=k$, then  $I_{k,3,d}$ is not Freiman by \cite[Theorem 3.3]{HZ}.

If $k+1\leq d\leq 2k-1$, then the elements  $x_1^{k-1}x_{2}^{d-k+1},x_1^{k}x_{2}^{d-k},x_1^{k}x_{2}^{d-k-1}x_{3},
x_1^{k-1}x_{2}^{d-k-1}x_{3}^{2},\\ x_1^{k-2}x_{2}^{d-k}x_{3}^{2},
x_1^{k-2}x_{2}^{d-k+1}x_{3}$  are the vertices of an induced  $6$-cycle of $G_{I_{k,3,d},s}$, and if
 $2k\leq d\leq 3k-3$, then the elements  $x_1^{k-1}x_{2}^{k}x_{3}^{d-2k+1},x_1^{k}x_{2}^{k-1}x_{3}^{d-2k+1},
x_1^{k}x_{2}^{k-2}x_{3}^{d-2k+2}, \\
x_1^{k-1}x_{2}^{k-2}x_{3}^{d-2k+3}, x_1^{k-2}x_{2}^{k-1}x_{3}^{d-2k+3},
x_1^{k-2}x_{2}^{k}x_{3}^{d-2k+2}$  are the vertices of an induced  $6$-cycle of $G_{I_{k,3,d},s}$.
Hence, in both cases,  $I_{k,3,d}$  is not Freiman by  Theorem \ref{main}.

(c) If $d=kn-1$,  then  $I_{k,n,d}=(x_1^{k-1}\prod\limits_{j=2}^{n}x_j^k,\ldots,x_i^{k-1}\prod\limits_{\begin{subarray}{l}
j=1\\
j\neq i
\end{subarray}}^{n}x_j^k,\ldots,x_n^{k-1}\prod\limits_{j=1}^{n-1}x_j^k)$. It follows that $G_{I_{k,n,d},s}$
is a complete graph. Hence $I_{k,n,d}$  is  Freiman.

It remains to show that $I_{k,3,d}$ is not Freiman if  $k+1\leq d\leq kn-2$. If $k+1\leq d\leq 2k$, then the elements  $x_1^{k}x_2^{d-k},x_1^{k}x_2^{d-k-1}x_4,x_1^{k-1}x_2^{d-k-1}x_3x_4$ and $x_1^{k-1}x_2^{d-k}x_3$  are the
vertices of an induced $4$-cycle of $G_{I_{k,n,d},s}$, if $2k+1\leq d\leq 3k-1$, then the elements  $x_1^{k}x_2^{k}x_3^{d-2k},x_1^{k}x_2^{k-1}x_3^{d-2k}x_4,
x_1^{k-1}x_2^{k-1}x_3^{d-2k+1}x_4,
x_1^{k-1}x_2^{k}x_3^{d-2k+1}$  are the vertices of an induced $4$-cycle of $G_{I_{k,n,d},s}$, and if  $3k\leq d\leq 4k-2$, then the elements  $x_1^{k}x_2^{k}x_3^{k-1}x_4^{d-3k+1},x_1^{k}x_2^{k-1}x_3^{k-1}x_4^{d-3k+2},
 x_1^{k-1}x_2^{k-1}x_3^{k}x_4^{d-3k+2},x_1^{k-1}x_2^{k}x_3^{k}x_4^{d-3k+1}$  are the vertices of an induced $4$-cycle of $G_{I_{k,n,d},s}$.

Now we consider the case   $(m-1)k-1\leq d\leq mk-2$ for any $5\leq m\leq n$. In this case the elements
$x_1^{k}x_2^{k}x_3^{k-1}x_4^{k-1}(\prod\limits_{j=5}^{m-1}x_j^k)x_m^{d-(m-1)k+2},
 x_1^{k}x_2^{k-1}x_3^{k-1}x_4^{k}(\prod\limits_{j=5}^{m-1}x_j^k)x_m^{d-(m-1)k+2},\\
  x_1^{k-1}x_2^{k-1}x_3^{k}x_4^{k}(\prod\limits_{j=5}^{m-1}x_j^k)x_m^{d-(m-1)k+2},
x_1^{k-1}x_2^{k}x_3^{k}x_4^{k-1}(\prod\limits_{j=5}^{m-1}x_j^k)x_m^{d-(m-1)k+2}$  are the vertices of an induced $4$-cycle of $G_{I_{k,n,d},s}$.
Thus, in all cases, $I_{k,n,d}$  is not Freiman.
\end{proof}

 \medskip
\hspace{-6mm} {\bf Acknowledgments}

 \vspace{3mm}
\hspace{-6mm}  This research is supported by the National Natural Science Foundation of China (No.11271275) and  by foundation of the Priority Academic Program Development of Jiangsu Higher Education Institutions.

\end{document}